\documentclass[10pt]{article}
\font\smallit=cmti10
\font\smalltt=cmtt10

\usepackage[fleqn]{nccmath}
\usepackage{amssymb,latexsym,epsfig,amsthm,chngpage,calc} 

\makeatletter

\renewcommand\section{\@startsection {section}{1}{\z@}
{-30pt \@plus -1ex \@minus -.2ex}
{2.3ex \@plus.2ex}
{\normalfont\normalsize\bfseries}}

\renewcommand\subsection{\@startsection{subsection}{2}{\z@}
{-3.25ex\@plus -1ex \@minus -.2ex}
{1.5ex \@plus .2ex}
{\normalfont\normalsize\bfseries}}

\renewcommand{\@seccntformat}[1]{\csname the#1\endcsname. }

\makeatother

\newtheorem{theorem}{Theorem}
\newtheorem{lemma}{Lemma}
\newtheorem{conjecture}{Conjecture}
\newtheorem{proposition}{Proposition}
\newtheorem{corollary}{Corollary}
\newtheorem{exm}{Example}[section]
\newtheorem{defn}{Definition}[section]

\begin{document}

\begin{center}
\uppercase{\bf  A window to the Convergence of a Collatz Sequence}
\vskip 20pt
{Maya Mohsin Ahmed}\\
{maya.ahmed@gmail.com}
\end{center}
\vskip 30pt
\centerline{\smallit Accepted: } 
\vskip 30pt

\centerline{\bf Abstract}
\noindent

In this article, we reduce the unsolved problem of convergence of
 Collatz sequences to convergence of Collatz sequences of odd numbers that
 are divisible by 3. We give an elementary proof of the fact that a
 Collatz sequence does not increase monotonically. We define a unique
 reverse Collatz sequence and conjecture that this sequence always
 converges to a multiple of $3$.

\pagestyle{myheadings} 
\markright{\smalltt INTEGERS: 14 (2014)\hfill} 
\thispagestyle{empty} 
\baselineskip=12.875pt 
\vskip 30pt

\section{Introduction} \label{intro}

The Collatz conjecture is an unsolved conjecture about the convergence
of a sequence named after Lothar Collatz.  A comprehensive study of
the Collatz conjecture can be found in \cite{lag1}, \cite{lag2},
and \cite{wir}. 

Consider the following operation  on a positive integer $A$.
If $A$ is even, divide it  by two. Otherwise, if $A$ is odd,
triple  it  and add  one.  Now,  form  a sequence  $c_i$  by
performing  this operation  repeatedly, beginning  with $A$,
and taking the result at each step as the input of the next:
\begin{align*}
c_i &= A,  \mbox{ if $i=0$;} \\ &= 3c_{i-1}+1,  \mbox{ if
  $c_{i-1}$ is odd;} \\ & =  c_{i-1}/2,  \mbox{ if $c_{i-1}$ is even.}
\end{align*} 
The sequence $c_i$ is called a {\em Collatz sequence} with {\em
  starting number} $A$. The {\em Collatz conjecture} says that this
sequence will eventually reach the number 1, regardless of which
positive integer is chosen initially. The sequence gets in to an
infinite cycle of 4, 2, 1 after reaching 1.

\begin{exm} \label{Collatzeg} {\em 
The Collatz sequence of $27$ is
\begin{gather*}
27, 82, 41, 124, 62, 31, 94, 47, 142, 71, 214, 107, 322, 161, 484,
242, 121, 364, 182, 91, 274, \\ 137, 412, 206, 103, 310, 155, 466,
233, 700, 350, 175, 526, 263, 790, 395, 1186, 593, 1780, \\ 890, 445,
 1336, 668, 334, 167, 502, 251, 754, 377, 1132, 566, 283, 850, 425,
1276, 638, 319, \\ 958, 479, 1438,  719, 2158, 1079, 3238, 1619, 4858,
2429, 7288, 3644, 1822, 911, 2734, 1367, \\ 4102, 2051, 6154,  3077, 9232,
4616, 2308, 1154, 577, 1732, 866, 433, 1300, 650, 325, 976, \\ 488, 244,
122, 61, 184, 92,  46, 23, 70, 35, 106, 53, 160, 80, 40, 20, 10, 5, 16,
8, 4, 2, 1,4,2,1,\\ 4,2,1, \dots
\end{gather*}
} \end{exm}
\noindent For the rest of this article, we will ignore the infinite cycle of
$4,2,1$, and say that a Collatz sequence {\em converges to $1$}, if it
reaches $1$. 

 In this article, we focus on the subsequence of odd numbers of a
 Collatz sequence. This is because every even number in a Collatz
 sequence has to reach an odd number after a finite number of steps.
 Observe that the Collatz conjecture implies that the subsequence of
 odd numbers of a Collatz sequence converges to $1$.
\begin{exm} \label{subseqodd27eg} {\em The subsequence of odd numbers of the Collatz sequence of $27$ in Example \ref{Collatzeg} is
\begin{gather*}
27, 41, 31, 47, 71, 107, 161, 121, 91, 137, 103, 155, 233, 175, 263,
395, 593, 445, 167, 251,  377, \\ 283,  425, 319, 479, 719, 1079, 1619,
2429, 911, 1367, 2051, 3077, 577, 433, 325, 61, 23, 35, 53, \\ 5, 1, 1,
\dots
\end{gather*} } \end{exm} 
This article is organized as follows.  In Section
\ref{equivconsedoddsection}, we show that to prove the Collatz
conjecture, it is sufficient to show that the Collatz sequences of
numbers that are multiples of $3$ converge. In Section
\ref{riseandfallsection}, we prove that a Collatz sequence cannot
increase monotonically. Finally, in Section
\ref{reversecollatzsection}, we define and discuss a unique reverse
Collatz sequence.

\section{Reducing the Collatz conjecture to multiples of 3.} \label{equivconsedoddsection}

In this section we show that to prove the Collatz conjecture, it is
sufficient to show that the Collatz sequences of numbers that are
multiples of $3$ converge to $1$.

Observe that a Collatz sequence of a number is unique.  However, a
Collatz sequence can be a subsequence of another Collatz sequence.  We
say two Collatz sequences are {\em equivalent} if the second odd
number occurring in the sequences are same. Consequently, two
equivalent Collatz sequences become equal from the second odd term
onwards.

\begin{exm} {\em
The Collatz sequence of 3 is
\[
3, 10, 5, 16, 8, 4, 2, 1, 1, \dots \]
The Collatz sequence of 13 is
\[
13, 40, 20, 10, 5, 16, 8, 4, 2, 1, 1, \dots 
\]
Observe that the two sequences merge at the odd number $5$.  Hence the
Collatz sequences of $3$ and $13$ are equivalent.  }\end{exm}

Let $e_n$ denote the recursive sequence $e_n = 4e_{n-1}+1$, such that,
 $n=0,1,2, \dots$ and $e_{-1}=0$. Observe that $e_n = \sum_{i=0}^n 4^i,
n \geq 0$.  Thus, $e_n$ is the sequence
\begin{gather*}
1, 5, 21, 85, 341, 1365, 5461, 21845, 87381, 349525, 1398101,  5592405, \dots  
\end{gather*}
These numbers play a pivotal role in the study of a Collatz
sequence.

\begin{lemma} \label{4n4n-1lemma}
For $n \geq 0$,
\[ 4^n = 3e_{n-1} + 1.
\]
\end{lemma}

\begin{proof}
 We use induction to  prove  this lemma. For $n=0$,
\[
1= 4^0  = 3e_{-1}+1 = 1.
\]
Hence the result is true for $n=0$. Assume the hypothesis is true for
$n=k$, that is $4^k = 3e_{k-1}+1$. Then for $n=k+1$, we get
\[  
4^{k+1} = 4 (4^k) = 4(3e_{k-1}+1) = 3 \times 4e_{k-1} + 3 + 1 = 3(4e_{k-1}+1)+1 =
3e_k+1.
\]
Therefore by induction the result is true for all $n$.
\end{proof}

Let $P$ be a positive number. We say $A$ is a {\em jump} of height
$n$ from $P$ if $A = 4^nP+e_{n-1}$.

\begin{exm} {\em
$13 = 4 \times 3 + 1$ is a jump from $3$ of height $1$. $53 = 4 \times
    13 +1 = 4^2 \times 3 +4 + 1$ is a jump from $13$ of height $1$ and
    a jump from $3$ of height $2$.  }
\end{exm}

In the next lemma, we consider a sequence of jumps from a given
number.

\begin{lemma} \label{bnthm}
Let $b_0$ be any positive integer. Consider the sequence $ b_i = 4b_{i-1}+1$.
For any positive integer $n$, $ 3b_n+1 = 4^n(3b_0+1)$.
\end{lemma}

\begin{proof}

\begin{align*}
b_1 &= 4b_0 + 1 \\
b_2 &= 4(4b_0+1) + 1 = 4^2b_0 + 4 + 1 \\
b_3 &= 4(4(4b_0+1) + 1)+1 = 4^3b_0 + 4^2 + 4+1 \\
&\mathrel{\makebox[\widthof{=}]{\vdots}} 
\end{align*}
Thus, we get $b_n = 4^nb_0 + e_{n-1}$. Therefore
\begin{align*}
3b_n + 1 & =   3(4^nb_0 + e_{n-1})+1 \\
& =  3\times4^nb_0 +(3e_{n-1}+1) \\ 
&=  3\times4^nb_0  + 4^n   \hfill \mbox{ (by Lemma \ref{4n4n-1lemma}, } 3e_{n-1}+1 = 4^n)\\
&= 4^n(3b_0+1). \end{align*} \end{proof}

In the next theorem we show that the set of odd numbers that can occur
consecutively before a given odd number $B$ in the odd subsequence of
a Collatz sequence are all jumps from a unique number derived from $B$.

\begin{theorem}  \label{consecuoddthm} 

Let $B$ be an odd number.

\begin{enumerate}
\item \label{nooddpartomo3} If $B \equiv 0$ mod $3$, then there are no
  odd numbers before $B$ in a Collatz sequence containing $B$.
\item  Let
\begin{ceqn}
\begin{equation} \label{thepeqn} P = \left \{
\begin{array}{ll}  \frac{4B-1}{3} & \mbox{if } B \equiv 1 \mbox{ mod } 3 \mbox{ 
and }\\ \\\frac{2B-1}{3} & \mbox{if } B \equiv 2 \mbox{ mod } 3. 
\end{array}
\right .
\end{equation}
\end{ceqn}

Let $A$ be an odd number. Then $A$ occurs before $B$, consecutively,
in the odd subsequence of a Collatz sequence containing $B$, if and
only if, $A$ is a jump of $P$ of height $n \geq 0$.  That is, $A \in
\{4^nP + e_{n-1}; n \geq 0 \}$. Thus, $P$ is the smallest odd integer
that can occur consecutively before $B$ in the odd subsequence of a
Collatz sequence containing $B$.

\end{enumerate}
\end{theorem}

\begin{proof}  \hfill

\begin{enumerate}

\item If there is an odd number before $B$ in a Collatz sequence, then
  there is some odd number $A$, consecutively, before $B$. Then by the
  definition of a Collatz sequence, $B$ can be written as
  $B=(3A+1)/2^r$ for some $r \geq 1$. Consequently, if $B \equiv 0$
  mod $3$, then $3A+1 \equiv 0$ mod $3$, which is impossible. Hence
  there are no odd numbers before $B$ in any Collatz sequence.

\item Since $A$ is an odd number such that $A$ and $B$ occur
  consecutively  in   the  odd  subsequence   of  a  Collatz
  sequence,  by definition  of  the Collatz  sequence, $A  =
  (2^jB-1)/3$ for  some $j \geq 1$. Consequently,  $j$ is an
  integer  such  that
\begin{ceqn}
 \begin{equation}  \label{conditioneqn}
    2^jB \equiv  1 \mbox{ mod  } 3. \end{equation} \end{ceqn}  
Since we  assume $A$  exists, $B \not  \equiv 0$ mod  $3$ by
part (\ref{nooddpartomo3}). First  consider $B \equiv 2$ mod
$3$. Since $j  \geq 1$ and since $2B \equiv  1$ mod $3$, for
(\ref{conditioneqn}) to  hold $j=2r+1$ for some  $r \geq 0$.
On the  other hand if  $B \equiv 1$  mod $3$, then  $2B \not
\equiv 1$  mod $3$, but  $4 B \equiv  1$ mod $3$.  Thus, for
(\ref{conditioneqn}) to hold, $j=2r+2$ for some $r \geq 0$.

Thus we conclude that for some $r \geq 0$,
\begin{equation} \nonumber A = \left \{
\begin{array}{ll}  \frac{4^{r+1}B-1}{3} & \mbox{if } B \equiv 1 \mbox{ mod } 3, \\ \\\frac{4^r \times 2B-1}{3} & \mbox{if } B \equiv 2 \mbox{ mod } 3. 
\end{array}
\right .
\end{equation}

We substitute $r=0$ in the above equation to get $P$.  

We now prove that $A$ is a jump from $P$ of height $r$. Since $P =
(2^k B -1)/3$ such that $k$ is either $1$ or $2$ depending on $B$,
\begin{align*}
4^rP+e_{r-1} &= 4^r \frac{2^kB-1}{3} + e_{r-1} \\
& =  \frac{4^r \times 2^kB - 4^r + 3e_{r-1}}{3} \\
&= \frac{4^r \times 2^kB - 4^r +  (3e_{r-1}+1) -1}{3}
 \mbox{  (adding and subtracting $1$) } \\ 
& =  \frac{4^r \times 2^kB - 4^r + 4^r -1}{3} \mbox{ (since $3e_{r-1}+1 = 4^r$)} \\
& =  \frac{4^r \times 2^kB -1}{3}  \\
&= A.
\end{align*}

Hence $A$ is a jump of height $r$ from $P$.

Now we prove the other direction, that is, if $A$ is a jump from $P$
of height $r$ for some $r \geq 0$, then $A$ occurs consecutively
before $B$ in the odd subsequence of a Collatz sequence.

Since $A = 4^r P + e_{r-1}$, we apply Lemma \ref{bnthm} with $b_0=P$ and
$A = b_r$ to get
\begin{ceqn}
\begin{align} \label{3bip1ean}
3A+1 &= 3b_r + 1 = 4^r(3b_0 + 1) = 4^r(3P + 1). 
\end{align}
\end{ceqn}
 Since $P = (2^k B -1)/3$ such that $k$ is either $1$ or $2$ depending
 on $B$, 
\begin{ceqn}
\begin{align} \label{3ap1ean}
(3P+1) &= 2^k B.
\end{align}
\end{ceqn}
From  (\ref{3bip1ean}) and (\ref{3ap1ean}), we get $3A + 1= 4^r
\times 2^k B$. Thus the odd number following $A$ in a Collatz
sequence is $(3A+1)/(4^r \times 2^k) = B$.
\end{enumerate} \end{proof}

\begin{exm}  \label{333exm} {\em
Consider the odd number $B=23$. Then by Theorem \ref{consecuoddthm}, 
\[
P = \frac{2 \times 23 -1}{3} = 15. 
\]
The  unique sequence of  odd numbers  that can  occur before
$23$ in a Collatz sequence is $b_i = 4b_{i-1}+1$, where $b_0
=  15$. The  Collatz sequences  of $b_0=15,  b_1 =  61$, and
$b_3=3925$ are given below (check  that the second odd number
in all these sequences is 23):
\begin{gather*}
15, 46, \underline{23}, 70, 35,
106, 53, 160, 80, 40, 20, 10, 5, 16, 8, 4, 2, 1;
\\ \\
61, 184, 92, 46, \underline{23}, 70, 35,
106, 53, 160, 80, 40, 20, 10, 5, 16, 8, 4, 2, 1;
\\ \\
3925, 11776, 5888, 2944, 1472, 736, 368, 184, 92, 46,
\underline{23}, 70, 35, 106, 53, 160, 80, 40, 20,\\ 10, 5, 16, 8, 4, 2,
1.
\end{gather*}
}
\end{exm}

\begin{corollary}
Let $B$ be an integer and let $B > 1$. If the Collatz sequence of $B$
converges then either $B=2^r$ for some $r \geq 1$ or the sequence
reaches $e_n$ for some $n \geq 1$.
\end{corollary}
\begin{proof} If $B=2^r$ for some $r \geq 1$ then clearly the Collatz
sequence of $B$ converges to $1$. On the other hand, if $B$ is not a
power of $2$, and since $B \neq 1$, the convergent Collatz sequence of
$B$ contains an odd number greater than $1$. Let $A > 1$ denote the
odd integer that occurs consecutively before $1$ in the Collatz
sequence of $B$. By Theorem \ref{consecuoddthm}, since $1 \equiv 1$
mod $3$, the smallest integer $P$ that can occur before $1$ in a
Collatz sequence is $P=(4\times 1-1)/3 = 1=e_0$. Since $A$ is greater
than $1$, $A$ is a jump from $P$ of height greater than or equal to
$1$.  Observe that $e_n$ is a jump of height $n$ from $1$, by
definition.  Consequently, $A=e_n$ for some $n \geq 1$.  \end{proof}

\begin{corollary}  \label{startingnumbercorollary} \hfill

\begin{enumerate}

\item \label{cruzfpreqiucseqpart} 

Let $B$ be an  odd number that is not divisible by $3$. Let 
\begin{align*} P = \left \{
\begin{array}{ll}  \frac{4B-1}{3} & \mbox{if } B \equiv 1 \mbox{ mod } 3
\\ \\\frac{2B-1}{3} & \mbox{if } B \equiv 2 \mbox{ mod } 3. 
\end{array}
\right . \end{align*}
Let $A$ be an odd number that occurs consecutively before $B$ in the
odd subsequence of a Collatz sequence. Consider the sequence $b_i =
4b_{i-1}+1$ with $b_0=P$.
\begin{enumerate}
\item \label{pseqequivpart} For any $i$, the Collatz sequence of $A$
  and $b_i$ are equivalent.

\item \label{smallp}  The Collatz sequence of $P$ is the sequence with smallest
starting odd number that is equivalent to the Collatz sequence of $A$.

\item \label{ciequivseqpart} Let $c_0=A$ and $c_i = 4c_{i-1}+1$. Then
  for any $i$, the Collatz sequence of $A$ and $c_i$ are equivalent.

\end{enumerate}

\item \label{startwithmultof3part} Let $A$ be an odd number.  The
  Collatz sequence of $A$ is equivalent to a Collatz
  sequence of a  number divisible by $3$.

\end{enumerate}
\end{corollary}
\begin{proof} \hfill
\begin{enumerate}

\item Part (\ref{pseqequivpart}) and part (\ref{smallp}) is a direct
  application of Theorem \ref{consecuoddthm}.  Again, by Theorem
  \ref{consecuoddthm}, $A = b_r$ for some $r \geq 0$. Hence $c_i$ is a
  subsequence of $b_i$. Therefore, for any $i$, the Collatz sequence
  of $A$ and $c_i$ are equivalent by part (\ref{pseqequivpart}).

\item The Collatz sequence with starting numbers $A$, $4A+1$, and
  $4(4A+1)+1$ are equivalent by part (\ref{ciequivseqpart}).  If $A
  \equiv 0$ mod $3$, there is nothing to prove. If $A \equiv 2$ mod
  $3$, then $4A+1 \equiv 0$ mod $3$.  On the other hand, if $A \equiv
  1$ mod $3$, then $4(4A+1)+1 \equiv 0$ mod $3$.  Hence the Collatz
  sequence of $A$ is equivalent to a Collatz
  sequence of a  number divisible by $3$.

\end{enumerate}
\end{proof}

\begin{exm} \hfill {\em
\begin{enumerate}

\item The Collatz sequences of $7$ and $a_r=4^r \times 7+e_{r-1}$, $r \geq 0$, are
  equivalent by part (\ref{ciequivseqpart}) of Corollary
  \ref{startingnumbercorollary}.

 The Collatz sequences of $a_0=7, a_1=29$, and $a_2=117$ are given below.
\begin{gather*}
7, 22, 11, 34, 17, 52, 26, 13, 40, 20, 10, 5, 16, 8, 4, 2, 1 \\ 
29, 88, 44, 22, 11, 34, 17, 52, 26, 13, 40, 20, 10, 5, 16, 8,,4, 2, 1 \\ 
117, 352, 176, 88, 44, 22, 11, 34, 17, 52, 26, 13, 40, 20, 10, 5, 16, 8, 4, 2, 1
\end{gather*}

\item Since $11$ comes right after $7$ in the Collatz sequence of $7$, we apply
  part (\ref{smallp}) of Corollary \ref{startingnumbercorollary} to
  derive $P=7$. Hence $7$ is the smallest starting number for this set
  of equivalent Collatz sequences.

\item The Collatz sequence of $7$ is equivalent to the Collatz
  sequence of $117$ which is divisible by $3$.
\end{enumerate}
}\end{exm} 

Since a multiple of $3$ is always the starting odd number of a Collatz
sequence, and every Collatz sequence is equivalent to a Collatz
sequence of a multiple of $3$, the convergence problem of Collatz
sequences reduces to the convergence of Collatz sequences of multiples of $3$.

\section{The rise and fall of a  Collatz sequence.} \label{riseandfallsection}

In this  section, we  give an elementary  proof of  the fact
that  a Collatz  sequence does  not end  in  a monotonically
increasing  sequence. For other  approaches to  this result,
see \cite{lag1}.

Let $A$ be an odd
number. Write $A$ in its binary form,
\begin{align*}  \begin{array}{lllllllllll}
A &=& 2^{i_1} + 2^{i_2} + \cdots + 2^{i_m}+ 2^n+2^{n-1} + 2^{n-2} + \cdots + 2^2+ 2+1,\\ 
& & \hfill  \mbox{such that } i_1 > i_2 > \cdots > i_m > n+1.
\end{array}
\end{align*} 
 The {\em tail} of $A$ 
is defined as $2^n+2^{n-1} + 2^{n-2} + \cdots + 2^2+ 2+1$. We call $n$ the {\em
  length of the tail} of $A$. 

\begin{exm}{\em
The tail of $27 = 2^4 + 2^3 + 2 + 1$ is $2+1$ and hence has length $1$,
the tail of $161 = 2^7 + 2^5 + 1$ is $1$ and hence is of length $0$,
and the tail of $31 = 2^4 + 2^3 + 2^2 + 2 + 1$ is the entire number
$2^4 + 2^3 + 2^2 + 2 + 1$ and therefore has length $4$.  }
\end{exm}

The following lemma is used to work with  the numbers written in binary
form in this section.

 \begin{lemma} \label{computationslemma} \hfill

\begin{enumerate}

\item \label{whathappensto2npart}  \begin{align*} \frac{3 \times 2^n}{2} = 2^n + 2^{n-1}. \end{align*} 

\item \label{sumof2npart} Let $0<i<n$,  then
\begin{equation} \label{eqnconsecumiddle}
 \frac{3}{2}(2^n + 2^{n-1} + 2^{n-2} + \cdots + 2^{n-i}) =
 2^{n+1} + (2^{n-1} + 2^{n-2} + \cdots + 2^{n-i+1}) + 2^{n-i-1}.
\nonumber
\end{equation}

\item \label{tailcomputations} Let $n > 0$, then
\begin{equation} \label{tailcompeqn}
\frac{3(2^{n}+2^{n-1} + \cdots + 2^2+2+1)+1}{2} = 2^{n+1}+2^{n-1} +
\cdots + 2^2+2+1. \nonumber
\end{equation}

\end{enumerate}
\end{lemma}

\begin{proof} \hfill

\begin{enumerate}

\item \begin{align*} 3 \times 2^n =  (2+1)\times2^n  = 2 \times 2^n + 2^n  = 2^{n+1} + 2^n. \end{align*}
Therefore $\frac{3 \times 2^n}{2} = 2^n + 2^{n-1}$.

\item  \begin{align*}
\frac{3}{2}(2^n + 2^{n-1} + 2^{n-2} + \cdots + 2^{n-i})  &= (2^n +
2^{n-1})+ (2^{n-1} + 2^{n-2}) + \cdots \\ & +(2^{n-i+1} +2^{n-i})  +
(2^{n-i}+2^{n-i-1})     \mbox{ (by part  (\ref{whathappensto2npart}))} 
\\ &= 2^n + (2^{n-1}+ 2^{n-1}) +
(2^{n-2}+ 2^{n-2})+ \cdots \\ &  + (2^{n-i}+ 2^{n-i})+2^{n-i-1} 
\mbox{ (rearranging terms)} \\ &= 2^n + (2^n) + (2^{n-1}) \cdots +
(2^{n-i+1})+2^{n-i-1} \\ &= 2^{n+1} + 2^{n-1} \cdots + 2^{n-i+1}+2^{n-i-1}.
\end{align*}

\item \begin{align*}
\frac{3(2^{n}+2^{n-1} + \cdots + 2^2+2+1)+1}{2} & = &
\frac{3(2^{n}+2^{n-1} + \cdots + 2^2+2)+ 3 +1}{2} \\ & = &
\left(\frac{3(2^{n}+2^{n-1} + \cdots + 2^2+2)}{2} \right ) + 2 \\ &
= & (2^{n+1}+2^{n-1} + \cdots + 2^2+1) + 2    \\ && (\mbox{by
  part (\ref{sumof2npart}) }) \\ & = & 2^{n+1}+2^{n-1} + \cdots +
2^2+2+1  \\ & &
\hfill (\mbox{rearranging terms}).
\end{align*}

\end{enumerate} \end{proof}

Let $A$ and $B$ be two odd numbers such that $B$ occurs,
consecutively, after $A$ in a Collatz sequence. If $B > A$, we say
that the Collatz sequence {\em increases} at $A$, otherwise, we say
that the Collatz sequence {\em decreases} at $A$. The next theorem
relates the tail of a number to the increase or decrease of a Collatz
sequence at that number.

\begin{theorem}  \label{tailsthm}
Let $A$ be an odd number and let $n$ denote the length of the tail of
$A$.  Let $a_i$ denote the sequence of odd numbers in the Collatz
sequence of $A$ with $a_0 = A$.
\begin{enumerate}

\item \label{tailis0thmpart} If $n \geq 1$, then  for some 
$i_1>i_2 > \cdots > i_m > n+1$,
\[ a_i = \frac{3a_{i-1}+1}{2} = \frac{3^i}{2^i} (2^{i_1} +
  2^{i_2} + \cdots + 2^{i_m}+2^{n+1}) -1, \mbox{ for } i=1, \dots
  n.
\]
 The length of the tail of $a_i$ is $n-i$. Hence the length of the
 tail of the $n$-th odd number after $A$ is $0$.

\item If $n=0$, then 
\[ a_1  = \frac{3A+1}{2^k}, k \geq 2. \]
\end{enumerate}
 
Moreover, if the Collatz sequence is
decreasing at $A$ then $n=0$ and  if  the Collatz
sequence is increasing at $A$, then  $n \geq 1$.
  
\end{theorem}

\begin{proof} \hfill 

\begin{enumerate}

\item 

Since the length of the tail of $A $ is $n \geq 1 $, for some 
\begin{gather*} 
i_1>i_2 > \cdots > i_m > n+1, \\
A = 2^{i_1} + 2^{i_2} + \cdots + 2^{i_m}+ 2^n+2^{n-1} + 2^{n-2} + \cdots + 2^2+ 2+1.
\end{gather*}
Consequently, $A \equiv 3$ mod $4$. Hence $3A+1 \equiv 2$ mod $4$
which implies $a_1=(3A+1)/2$. Since $a_1 > A$ we conclude that the
Collatz sequence is increasing at $A$.  

On the other hand, if $a_1=(3A+1)/2$, then  $A \equiv 3$ mod $4$, which
implies $n \geq 1$. Therefore,  $n \geq 1$ if and only if the Collatz
sequence is increasing at $A$.
  
In the binary form,
\begin{ceqn}
\begin{equation} \label{a1manipulateqn}
a_1 = \frac{3A+1}{2} = \frac{3(2^{i_1} + 2^{i_2} + \cdots + 2^{i_m})}{2} + 
\frac{3(2^{n}+2^{n-1} + \cdots + 2^2+2+1)+1}{2}.
\end{equation}
\end{ceqn}

By part (\ref{whathappensto2npart}) of Lemma  \ref{computationslemma},
\[
 \frac{3}{2}(2^{i_1} + 2^{i_2}  + \cdots + 2^{i_m}) =
2^{i_1} + 2^{i_1-1} + 2^{i_2} +  2^{i_2-1} + \cdots + 2^{i_m} + 2^{i_m-1}.
\]

By part (\ref{tailcomputations}) of Lemma  \ref{computationslemma},
\[
\frac{3(2^{n}+2^{n-1} + \cdots + 2^2+2+1)+1}{2} = 2^{n+1}+2^{n-1} +
\cdots + 2^2+2+1.
\]

Combining the two results, (\ref{a1manipulateqn}) becomes 
\begin{equation*} 
a_1 = 2^{i_1} + 2^{i_1-1}  + \cdots + 2^{i_m} + 2^{i_m-1}+
 2^{n+1}+2^{n-1} +
\cdots + 2^2+2+1.
\end{equation*}

Since $i_m-1 > n$, the tail of $a_1$ is $2^{n-1} + 2^{n-2} + \cdots +
2^2 + 2 + 1$. Thus the length of the tail of $a_1$ is $n-1$.

Since $n-j \geq 1$ for $j < n$, we repeat the above argument and
derive that $a_i = (3a_{i-1}+1)/2$ and the length of the tail of $a_i$
is $n-i$ for $1<i \leq n$. Hence the length of the tail of the $n$-th
odd number that occurs after $A$ in the Collatz sequence is $0$.

We use induction to prove, \[a_i = \frac{3^i}{2^i}
(2^{i_1} + 2^{i_2} + \cdots + 2^{i_m}+2^{i+1}) -1 \mbox{ for } i=1, \dots n. \]

When $i=1$,
\begin{align*} a_1 &=  \frac{3}{2}(2^{i_1} + 2^{i_2} + \cdots + 2^{i_m} + 2^{n}+2^{n-1} + \cdots + 2^2+2)+ 2 & \hfill (\mbox{by (\ref{a1manipulateqn})})\\ &=   \frac{3}{2}(2^{i_1} + 2^{i_2} + \cdots + 2^{i_m} + 2^{n}+2^{n-1} + \cdots + 2^2+2)+ (3-1) \\
 &= \frac{3}{2}(2^{i_1} + 2^{i_2} + \cdots + 2^{i_m} + 2^{n}+2^{n-1}
  + \cdots + 2^2+(2+2)) - 1 \\ &= \frac{3}{2}(2^{i_1} + 2^{i_2} +
  \cdots + 2^{i_m} + 2^{n}+2^{n-1} + \cdots + (2^2+ 2^2)) - 1 \\  &= \frac{3}{2}(2^{i_1} + 2^{i_2} +
  \cdots + 2^{i_m} + 2^{n}+2^{n-1} + \cdots + (2^3+ 2^3)) - 1 \\ &\mathrel{\makebox[\widthof{=}]{\vdots}} \\
 &=
  \frac{3}{2}(2^{i_1} + 2^{i_2} + \cdots + 2^{i_m} + (2^n+2^n))-1 \\
 &=
  \frac{3}{2}(2^{i_1} + 2^{i_2} + \cdots + 2^{i_m} + 2^{n+1})-1.
\end{align*}
Assume $k<n$, $a_k = \frac{3^k}{2^k}(2^{i_1} + 2^{i_2} + \cdots +
2^{i_m} + 2^{n+1})-1$, then
\begin{align*}
a_{k+1} &= \frac{3a_{k-1}+1}{2}\\  &= \frac{1}{2} \left ( 3 \left
( \frac{3^k}{2^k}(2^{i_1} + 2^{i_2} + \cdots + 2^{i_m} + 2^{n+1})-1
\right ) +1\right ) \\  &= \frac{1}{2} \left (
\frac{3^{k+1}(2^{i_1} + 2^{i_2} + \cdots + 2^{i_m} + 2^{n+1})-3 \times
  2^k+2^k}{2^k} \right ) \\  &= \frac{3^{k+1}(2^{i_1} + 2^{i_2} +
  \cdots + 2^{i_m} + 2^{n+1})- 2^{k+1}}{2^{k+1}}  \\ 
 &= \frac{3^{k+1}}{2^{k+1}}(2^{i_1} + 2^{i_2} +
  \cdots + 2^{i_m} + 2^{n+1})- 1.
\end{align*}
Hence proof follows by induction.

\item If the tail of $A$ is $1$, then $A$ is of the form $A = 2^{i_1}
  + \cdots + 2^{i_m} + 1$, where $i_m > 1$. Therefore $A \equiv 1$ mod
  $4$, which implies $3A+1 \equiv 0$ mod $4$.  Therefore $a_1 =
  (3A+1)/2^k$, where $k \geq 2$. Consequently, $a_1 < A$. Therefore,
  the Collatz sequence is decreasing at $A$.

On the other hand, if $a_1=(3A+1)/2^k$, where $k \geq 2$, then $A
\equiv 1$ mod $4$, which implies $n = 0$. Thus $n=0$, if and only if,
the Collatz sequence is decreasing at $A$.
\end{enumerate}
\end{proof}

\begin{exm} {\em 
Let $a_i$ denote the sequence of odd numbers in a Collatz
sequence, where $a_0 = 319$. Since 
$a_0=319 = 2^8 + 2^5 + 2^4 + 2^3 + 2^2 + 2 + 1 $ has tail of length $5$,
by Theorem \ref{tailsthm},
\[ a_i= \frac{3^i}{2^i} (2^8 + 2^6) - 1
\mbox{ for } i = 1, \dots, 5. \]
Thus,
\begin{align*}
a_{1} &= \frac{3}{2} (2^8 + 2^6) - 1 = 479; \\
a_{2} &= \frac{3^2}{2^2} (2^8 + 2^6) - 1 = 719; \\
a_{3} &= \frac{3^3}{2^3} (2^8 + 2^6) - 1 = 1079; \\ 
a_{4} &= \frac{3^4}{2^4} (2^8 + 2^6) - 1 = 1619; \\
 a_{5} &= \frac{3^5}{2^5} (2^8 + 2^6) - 1 = 2429.
\end{align*}
Since
\[  a_{5} = 2429 = 2^{11}+2^8+2^6+2^5+2^4+2^3+2^2+1, \]
 the length of the tail of $a_5$ is zero as expected.

}\end{exm}

Finally we conclude that every rise of a Collatz sequence is followed
by a fall in the next corollary.

\begin{corollary}
A Collatz sequence does not increase monotonically.
\end{corollary}
\begin{proof}  Let $n$ denote the length of the tail of $A$.  By
Theorem \ref{tailsthm}, a Collatz sequence increases at $A$ if and
only if the length of the tail of $A$ is $n \geq 1$.  By Theorem
\ref{tailsthm}, the $n$-th odd number, occurring after $A$, say $B$,
has tail of length $0$. Hence the Collatz sequence decreases at $B$.
Therefore a Collatz sequence cannot increase monotonically.  \end{proof}



\section{The Reverse Collatz sequence.} 
\label{reversecollatzsection}

Given an odd number $A$, we know from Section
\ref{equivconsedoddsection} that there is a unique sequence of odd
numbers that can occur before $A$ in a Collatz sequence. We restrict
ourselves to the smallest number of this sequence.  This gives us the
desired uniqueness necessary to get a well defined sequence in the
reverse direction defined as follows.

\begin{defn} Reverse Collatz Sequence: Let $A$ be a positive  integer. 
\[
r_i = \left \{  \begin{array}{llllllll} A  & \mbox{ if $i = 0$; } \\

 \frac{r_{i-1}-1}{3} & \mbox{ if $r_{i-1} \equiv 1$ mod $3$ and $r_{i-1}$ is
  even}; \\ 
2r_{i-1} &  \mbox{ if $r_{i-1} \not \equiv 1$ mod $3$ and
    $r_{i-1}$ is even;} \\ 2r_{i-1} &  \mbox{ if $r_{i-1}$ is odd}.
\end{array}
\right .
\]
\end{defn}

\begin{exm} \label{reveresecollarzexamp} \hfill {\em 
\begin{enumerate}
\item  \label{reversecollega} The reverse Collatz sequence with starting number  $121$ is :
\begin{gather*}
121, 242, 484, 161, 322, 107, 214, 71, 142, 47, 94, 31, 62, 124, 41,
82, 27, 54, 108, 216, \dots
\end{gather*}
Observe that since $27$ is divisible by $3$, all the other numbers
after $27$ are of the form $2^i \times 27$ for some $i$.

\item \label{reversecolleg4ap1} The reverse Collatz sequence with
  starting number $485$ is :
\begin{gather*}
485, 970, 323, 646, 215, 430, 143, 286, 95, 190, 63, 126, \dots
\end{gather*}
\end{enumerate}
}
\end{exm}

 We say that a reverse Collatz sequence {\em converges} if the
 subsequence of odd numbers of the sequence converges to a multiple of
 $3$. Recall that apart from the starting number, a Collatz sequence of
 an odd number cannot contain multiples of $3$ by part
 (\ref{nooddpartomo3}) of Theorem \ref{consecuoddthm}.

\begin{exm} {\em  The Reverse Collatz
 sequence of $121$ (see Example \ref{reveresecollarzexamp}) converges
 because its subsequence of odd numbers
\[
121, 161, 107, 71, 47, 31,  41, 27
\]
converges to $27$. }
\end{exm}

Henceforth, when a reverse Collatz sequence converges to $t$, we will
ignore the terms $2^it$ of the sequence.  

In the next step, we look at some properties of a reverse Collatz
sequence. Let $A$ be an odd number. We say that the reverse Collatz
sequence of $A$ is {\em trivial} if $A$ is the only odd number in its
reverse Collatz sequence.
\begin{proposition} \label{somerevcollprop}
Let $A$ be an odd number.
\begin{enumerate}
\item \label{trivialrevcollpart} If $A$ is divisible by $3$ then the reverse Collatz sequence of $A$ is
 trivial.

\item \label{reversepicompart}
 Let $p_i$ denote the subsequence of odd
  numbers in the reverse Collatz sequence of $A$. If $p_i \equiv 0$
  mod $3$, then $p_{i+1}$ do not exist. Otherwise, $p_{i+1}$ is the
  smallest odd number before $p_i$ in any Collatz sequence and
\begin{ceqn}
\begin{equation} \label{reversecollapieqn}
p_{i+1} = \left \{ \begin{array}{l}
\frac{2p_i-1}{3} \mbox{ if } p_i \equiv 2 \mbox{ mod } 3 \\ \\
\frac{4p_i-1}{3} \mbox{ if } p_i \equiv 1 \mbox{ mod } 3 
\end{array}
\right .
\end{equation}
\end{ceqn}
\item The only odd number that can be a jump from another odd number
  in a reverse Collatz sequence is the starting number of the reverse
  Collatz sequence.

\end{enumerate}
\end{proposition}

\begin{proof} \hfill
\begin{enumerate}
\item Since there are no odd numbers before $A$ in a Collatz sequence
  by part (\ref{nooddpartomo3}) of Theorem \ref{consecuoddthm}, it
  follows that $A$ is the only odd number in its reverse Collatz
  sequence.

\item  Since $p_{i+1}$  is  the smallest  odd number  before
  $p_i$    in    a     Collatz    sequence,    by    Theorem
  \ref{consecuoddthm}, if $p_i \not  \equiv 0$ mod $3$, then
  $p_{i+1}$ is as  in (\ref{reversecollapieqn}), and if $p_i
  \equiv 0$ mod $3$, then $p_{i+1}$ does not exist,

\item By definition of the reverse Collatz sequence, $p_{i+1}$ are
  smallest possible odd numbers. Hence, by Theorem
  \ref{consecuoddthm}, the only odd number that can be a jump in a
  reverse Collatz sequence is the starting number of the reverse
  Collatz sequence.

\end{enumerate}
\end{proof}

\begin{lemma} \label{multiply2reversetill4thm}
Let $A$ be an odd number and let $p_i$  denote the
subsequence of odd numbers in the reverse Collatz sequence of $A$. Then,
for some  $k \geq 1$,  $p_k \not \equiv 2$ mod
  $3$. 
\end{lemma}

\begin{proof}

If $p_i \equiv 2$ mod $3$, then $p_{i+1} = (2p_{i-1}-1)/3$ by
  part (\ref{reversepicompart}) of Proposition
  \ref{somerevcollprop}. Therefore $p_{i+1} < p_i$.  If all the odd
  terms of the reverse Collatz sequence of $A$ are congruent to $2$
  mod $3$ then the sequence is strictly decreasing. The smallest such
  odd number that such a sequence can reach is $5$ in which case the
  next number in the sequence is $3$, a contradiction. Hence, there is
  a $k$ such that $p_k$ is  not congruent to $2$ mod $3$.
\end{proof}

By  Lemma \ref{multiply2reversetill4thm}, a  reverse Collatz
sequence reaches a  number not congruent to $2$  mod $3$. We
believe but cannot prove that  it reaches a number that is a
multiple of $3$.

\begin{conjecture}
The reverse Collatz sequence converges to a multiple of $3$ for every
number greater than one.
\end{conjecture}

The reverse Collatz sequence of a number $A$ is unique and hence it
provides for the definition of a unique {\em complete Collatz
  sequence} of $A$, namely the union of the reverse Collatz sequence
of $A$ and the Collatz sequence of $A$.
\begin{exm} {\em 
The complete Collatz sequence of $485$ is 
\begin{gather*}
63, 190, 95, 286, 143, 430, 215, 646, 323, 970, \underline{{\bf 485}},
1456, 728, 364, 182, 91, 274, 137, 412, 206, 103, \\ 310, 155, 466,
233, 700, 350, 175, 526, 263, 790, 395, 1186, 593, 1780, 890, 445,
1336, 668, 334, 167, \\ 502, 251, 754, 377, 1132, 566, 283, 850, 425,
1276, 638, 319, 958, 479, 1438, 719, 2158,,1079, 3238, \\ 1619, 4858,
2429, 7288, 3644, 1822, 911, 2734, 1367, 4102, 2051, 6154, 3077, 9232,
4616, 2308, 1154,\\ 577, 1732, 866, 433, 1300, 650, 325, 976, 488,
244, 122, 61, 184, 92, 46, 23, 70, 35, 106, 53, 160, 80, 40, \\ 20,
10, 5, 16, 8, 4, 2, 1.
\end{gather*}
 }\end{exm}

\end{document}